\documentclass{amsart}
\usepackage{amsmath}
\usepackage{paralist}
\usepackage{amsfonts}
\usepackage{amssymb}
\usepackage{amsthm}
\usepackage{amscd}
\usepackage{float}
\usepackage{tikz}
\usepackage{graphicx}
\usepackage[colorlinks=true]{hyperref}
\hypersetup{urlcolor=blue, citecolor=red}
\usepackage{hyperref}

\newtheorem{theorem}{Theorem}[section]

\newtheorem{proposition}[theorem]{Proposition}

\theoremstyle{definition}
\newtheorem{definition}[theorem]{Definition}

\newcommand{\T}{\mathbb{T}}
\newcommand{\R}{\mathbb{R}}
\newcommand{\Z}{\mathbb{Z}}

\newcommand{\C}{\mathbb{C}}

\author{Robert Schippa}
\title[Energy-critical NLS in modulation spaces]{Infinite-energy solutions to energy-critical nonlinear Schr\"odinger equations in modulation spaces}
\address{Karlsruhe Institute of Technology\\
Englerstrasse 2 \\
76131 Karlsruhe \\
Germany
}
\email{robert.schippa@kit.edu}

\begin{document}

\begin{abstract}
We prove new well-posedness results for energy-critical nonlinear Schr\"odinger equations in modulation spaces. This covers initial data with infinite mass and energy. The proof is carried out via bilinear refinements and adapted function spaces.
\end{abstract}

\maketitle

\section{Introduction}

In this paper we continue the study of modulation spaces as initial data for nonlinear Schr\"odinger equations started in \cite{Schippa2022}.
Modulation spaces in the present context are used to model initial data, which are decaying slower than functions in $L^2$-based Sobolev spaces. These spaces are natural because of their invariance under the linear Schr\"odinger evolution in sharp contrast with the $L^p$-based Sobolev spaces for $p \neq 2$. Modulation spaces were introduced by Feichtinger \cite{Feichtinger1983}; see also subsequent joint works with Gr\"ochenig \cite{FeichtingerGroechenig1989,FeichtingerGroechenig1989II,Groechenig2001}. The body of literature on modulation spaces is already huge, so we refer to \cite{Schippa2022,Chaichenets2018,WangHuoHaoGuo2011,BenyiOkoudjou2020} and references therein for an overview with an emphasis on the use of modulation spaces in the context of dispersive equations.

In the work \cite{Schippa2022} $L^p$-smoothing estimates in modulation spaces were considered:
\begin{equation}
\label{eq:LinearSmoothing}
\| e^{it \Delta} u_0 \|_{L^p([0,1],L^p(\R^d))} \lesssim \| u_0 \|_{M^s_{p,2}(\R^d)}.
\end{equation}

These turned out to be useful to prove well-posedness results for the cubic NLS
\begin{equation}
\label{eq:CubicNLS}
\left\{ \begin{array}{cl}
i \partial_t u + \Delta u &= \pm |u|^2 u, \quad (t,x) \in \R \times \R, \\
u(0) &= u_0 \in M^s_{p,2}(\R).
\end{array} \right.
\end{equation}
The solution was placed in Strichartz spaces, in which the linear part was estimated by \eqref{eq:LinearSmoothing} and the nonlinear part was iterated with Strichartz estimates.

By frequency localization and rescaling arguments, the estimates \eqref{eq:LinearSmoothing} followed from $\ell^2$-decoupling for the paraboloid due to Bourgain--Demeter \cite{BourgainDemeter2015}. Let $\mathcal{E}$ denote the Fourier extension operator for the (truncated) paraboloid:
\begin{equation*}
\mathcal{E} f(t,x) = \int_{\{ \xi \in \R^d : |\xi| < 1 \} } e^{i(x.\xi + t|\xi|^2)} f(\xi) d\xi.
\end{equation*}
Bourgain--Demeter proved the following estimates, which are sharp up to the $\varepsilon$-loss:
\begin{equation}
\label{eq:l2Decoupling}
\| \mathcal{E} f \|_{L^p(B_{d+1}(0,R))} \lesssim_{\varepsilon} R^{s_{\text{dec}}+\varepsilon} \big( \sum_\sigma \| \mathcal{E} f_\sigma \|_{L^p(w_{B_{d+1}(0,R)})}^2 \big)^{1/2}
\end{equation}
with $s_{\text{dec}}=s_{\text{dec}}(p,d)$ given by
\begin{equation*}
s_{\text{dec}}=
\begin{cases}
0, \qquad &2 \leq p \leq \frac{2(d+2)}{d}, \\
\frac{d}{4} - \frac{d+2}{2p}, &\frac{2(d+2)}{d} \leq p \leq \infty,
\end{cases}
\end{equation*}
and $f_\sigma$ denotes $f \cdot 1_{B(x_\sigma,R^{-1/2})}$ such that the family of $R^{-1/2}$-balls are finitely overlapping. In \cite{Schippa2022} was pointed out how the right-hand side is related to the modulation space norm of the initial value by rescaling and a kernel estimate. Thus, \eqref{eq:l2Decoupling} indeed gives \eqref{eq:LinearSmoothing} with $s > 2s_{\text{dec}}(p,d)$. It was also shown in \cite{Schippa2022} that $s \geq 2s_{\text{dec}}(p,d)$ is necessary for \eqref{eq:LinearSmoothing} to hold true.
\begin{theorem}[{$L^p$-smoothing estimates in modulation spaces,~\cite[Theorem~1.1]{Schippa2022}}]
\label{thm:LinearSmoothing}
Let $d \geq 1$ and $p \geq 2$. Then, \eqref{eq:LinearSmoothing} holds true for $s>2s_{\text{dec}}(p,d)$.
\end{theorem}

 In the present work, we show bilinear refinements via Galilean invariance:
\begin{proposition}
\label{prop:BilinearRefinement}
Let $d \geq 3$, $\varepsilon > 0$, and $N_1, N_2 \in 2^{\mathbb{N}_0}$ with $N_2 \lesssim N_1$. Then, we find the following estimate to hold:
\begin{equation}
\label{eq:BilinearEstimate}
\| P_{N_1} e^{it \Delta} f_1 P_{N_2} e^{it \Delta} f_2 \|_{L^2_{t,x}([0,1] \times \R^d)} \lesssim_\varepsilon N_2^{\frac{d-2}{2}+\varepsilon} \| P_{N_1} f_1 \|_{M_{4,2}(\R^d)} \| P_{N_2} f_2 \|_{M_{4,2}(\R^d)}.
\end{equation}
\end{proposition}
\noindent Bilinear refinements go back to Bourgain \cite{Bourgain1993,Bourgain1998}.

\medskip

Next, we apply bilinear Strichartz estimates in modulation spaces to extend the local well-posedness theory of nonlinear Schr\"odinger equations.
We consider the energy-critical nonlinear Schr\"odinger equation for $d \in \{3,4\}$:
\begin{equation}
\label{eq:EnergyCriticalNLS}
\left\{ \begin{array}{cl}
i \partial_t u + \Delta u &= \pm |u|^{\frac{4}{d-2}} u \quad (t,x) \in \R \times \R^d, \\
u(0) &= u_0 \in M^{1+\varepsilon}_{4,2}(\R^d).
\end{array} \right.
\end{equation}
The equation \eqref{eq:EnergyCriticalNLS} is energy critical because the scaling
\begin{equation*}
u(t,x) \to \lambda^{\frac{d-2}{2}} u(\lambda^2 t, \lambda x)
\end{equation*}
leaves the energy invariant:
\begin{equation*}
E[u] = \int_{\R^d} \frac{|\nabla u(t,x)|^2}{2} \pm \frac{d-2}{d+2} |u|^{\frac{d+2}{d-2}} dx.
\end{equation*}
The corresponding scaling critical Sobolev space is $\dot{H}^1(\R^d)$. For local well-posedness in $\dot{H}^1(\R^d)$ we refer to the survey by Bourgain \cite{Bourgain1999Book}. Global well-posedness and scattering for the defocusing case is much harder and was proved for $d=3$ by Colliander \emph{et al.} \cite{CollianderKeelStaffilaniTakaokaTao2008} and for $d=4$ by Ryckman--Vi\c{s}an \cite{RyckmanVisan2007}; see also references therein and Bourgain's seminal contribution \cite{Bourgain1999} in the radially symmetric case. Sharp conditions for global well-posedness and scattering of the focusing equation in the radial case were proved by Kenig--Merle \cite{KenigMerle2006}. We refer to the solutions with initial data in $M^{1+\varepsilon}_{4,2}$ as infinite energy solutions because we can find a sequence of Schwartz functions $(f_n) \subseteq \mathcal{S}(\R^d)$ with $\| f_n \|_{M^{1+\varepsilon}_{4,2}} \leq C$, but $\| f_n \|_{H^1} \uparrow \infty$. For this purpose, consider a mollified indicator function $f_n = \chi_1 * 1_{B(0,n)}$ supported in $B(0,n+1)$, which has Fourier support rapidly decaying off $B(0,2)$. We normalize $1= \| f_n \|_{L^4(\R^d)} \sim \| f_n \|_{M^{1+\varepsilon}_{4,2}(\R^d)}$. But then,
\begin{equation*}
\| f_n \|_{H^1(\R^d)} \sim \| f_n \|_{L^2(\R^d)} \uparrow \infty.
\end{equation*}
The reason is that we allow for $L^4$-admissible decay at infinity in $M^{1+\varepsilon}_{4,2}$, which is slower than for $L^2$-based Sobolev spaces.

 Previous results on infinite energy solutions to nonlinear Schr\"odinger equations are due to Braz e Silva \emph{et al.} \cite{BrazSilva2009} with initial data in weak $L^p$-spaces. The results in \cite{BrazSilva2009} do not cover the energy critical equations though; see also \cite{CazenaveVegaVilela2001}. Moreover, weak $L^p$-spaces are not invariant under the linear propagation in contrast with modulation spaces. The first results on infinite energy solutions are due to Cazenave--Weissler \cite{CazenaveWeissler1998}, who consider initial data with finite linear solution in a certain $L^p$-norm. The results in \cite{CazenaveWeissler1998} do not cover the energy critical case. $L^2$-based Besov spaces were considered by Planchon \cite{Planchon2000}. We also mention the recent contributions of Correia \emph{et al.} \cite{Correia1,Correia2}.

Moreover, we remark how the arguments of \cite{Schippa2022} extend to $L^2$-critical equations for $d \in \{1,2\}$, i.e., the quintic NLS on the real line or the cubic NLS in $\R^2$. Note that
\begin{equation*}
H^s(\R^d) \sim M^s_{2,2}(\R^d) \hookrightarrow M_{p,2}^s(\R^d)
\end{equation*}
for $p \geq 2$ and $s \geq 0$. In this sense, the following well-posedness results are almost critical:
\begin{theorem}
\label{thm:AlmostL2Critical}
Let $s>0$ and $T>0$. 
\begin{itemize}
\item[(1)] Then, the equation
\begin{equation}
\label{eq:QuinticNLS1d}
\left\{ \begin{array}{cl}
i \partial_t u + \Delta u &= \pm |u|^4 u, \quad (t,x) \in \R \times \R, \\
u(0) &= u_0 \in M^s_{6,2}(\R) + L^2(\R)
\end{array} \right.
\end{equation}
is locally well-posed in $X_T = C([0,T],L^2(\R) + M^s_{6,2}(\R)) \cap L_t^6([0,T],L^6(\R))$ provided that $\| u_0 \|_{M^s_{6,2}(\R) + L^2(\R)} \leq \varepsilon(T)$.
\item[(2)] The equation
\begin{equation}
\label{eq:CubicNLS2d}
\left\{ \begin{array}{cl}
i \partial_t u + \Delta u &= \pm |u|^2 u, \quad (t,x) \in \R \times \R^2, \\
u(0) &= u_0 \in M^s_{4,2}(\R^2) + L^2(\R^2)
\end{array} \right.
\end{equation}
is locally well-posed in $X_T = C([0,T],L^2(\R^2)+ M^s_{4,2}(\R^2)) \cap L_t^4([0,T],L^4(\R^2))$ provided that $\|u_0 \|_{M^s_{4,2}(\R^2) + L^2(\R^2)} \leq \varepsilon(T)$.
\end{itemize}
\end{theorem}
Note how above we choose the the size of the initial data in terms of the existence time. It would be more practical to consider $T=T(u_0)$, which is possible, but not detailed for simplicity of presentation (see the proof of Theorem \ref{thm:LWPEnergyCriticalNLS} below).

 For $d \geq 3$ the derivative loss in the high frequencies of the $L^4$-Strichartz estimate has to be ameliorated via bilinear estimates. We show the following:
\begin{theorem}
\label{thm:LWPEnergyCriticalNLS}
Let $d \in \{3,4 \}$, and $\varepsilon > 0$. Then \eqref{eq:EnergyCriticalNLS} is analytically locally well-posed in $X_T \hookrightarrow C([0,T],M^{1+\varepsilon}_{4,2}(\R^d))$. For any $u_0 \in M^{1+\varepsilon}_{4,2}(\R^d)$ there is $T=T(u_0)$ such that there is a unique solution $u \in X_T$ to \eqref{eq:EnergyCriticalNLS}, and the data-to-solution mapping analytically depends on the initial value.
\end{theorem}
The first local well-posedness results on energy critical nonlinear Schr\"odinger equations in the periodic setting are due to Herr--Tataru--Tzvetkov \cite{HerrTataruTzvetkov2011,HerrTataruTzvetkov2014}. In these works, improved bilinear or trilinear estimates were proved via orthogonality in time.
This proof was simplified by Killip--Vi\c{s}an \cite{KillipVisan2016}, which is transferred to modulation spaces presently. Killip--Vi\c{s}an pointed out how the bilinear refinements can be used to show the well-posedness result for the energy critical equation. A few remarks on global results in the periodic setting are in order: Herr--Tataru--Tzvetkov \cite{HerrTataruTzvetkov2011} proved global well-posedness for small initial data by energy conservation. Since the Sobolev embedding $H^1(\T^d) \hookrightarrow L^{\frac{d+2}{d-2}}(\T^d)$ is sharp, the straight-forward use of energy conservation requires smallness of the $H^1(\T^d)$ norm. Ionescu--Pausader \cite{IonescuPausader2012} subsequently proved global well-posedness for large initial data in the defocusing case for $d=3$ (see also \cite{Strunk2015,Strunk2015Thesis}). But the global results fundamentally build on energy conservation, which is not at disposal for initial data in $M^s_{4,2}(\R^d)$, since these possibly have infinite mass and energy. Thus, global results, even in the defocusing case remain open for initial data in $M^s_{4,2}(\R^d)$. On the other hand, the classical blow-up arguments (cf. \cite{KenigMerle2006}) in the focusing case show that global solutions need not exist, if the energy is negative. For ill-posedness results for the nonlinear Schr\"odinger equation with initial data in modulation spaces we refer to Bhimani--Carles \cite{BhimaniCarles2020}.

For further reading, we also refer to the very recent contribution \cite{ChenHolmer2022}, in which \emph{unconditional uniqueness} of solutions in $C([0,T],H^1(X^d))$ for energy critical Schr\"o\-din\-ger equations is proved for $d \in \{3,4\}$ and $X \in \{\T,\R\}$. Chen and Holmer \cite{ChenHolmer2022} use the Gross--Pitaevskii hierarchy, which was previously considered by Herr and Sohinger \cite{HerrSohinger2019} to cover the whole subcritical range for $d=4$; see also references therein.

\bigskip

\emph{Outline of the paper.} In Section \ref{section:Notations} we recall basic facts about modulation spaces, and we introduce the function spaces used in the proof of Theorem \ref{thm:LWPEnergyCriticalNLS}. In Section \ref{section:Bilinear} we show Proposition \ref{prop:BilinearEstimate}, by which we prove Theorem \ref{thm:LWPEnergyCriticalNLS} in Section \ref{section:ProofEnergyCritical}. Theorem \ref{thm:AlmostL2Critical} is proved in Section \ref{section:ProofEnergyCritical} with linear Strichartz estimates for comparison.

\section{Preliminaries}
\label{section:Notations}
\subsection{Modulation spaces}
The modulation spaces $M^s_{p,q}(\R^d)$ for $d \geq 1$, $s \in \R$, $p,q \in [1,\infty]$ are defined through an isometric decomposition in Fourier space. Let $(\sigma_k)_{k \in \Z^d}$ with $\sigma_k = \sigma( \cdot - k)$ and $\sigma \in C^\infty_c(B(0,1))$ denote a smooth partition of unity. We define
\begin{equation*}
M^s_{p,q}(\R^d) = \{ f \in \mathcal{S}'(\R^d): \| f \|_{M^s_{p,q}(\R^d)} = \big\| \big( \langle k \rangle^{s} \| \sigma_k(D) f \|_{L^p(\R^d)} \big)_{k \in \Z^d} \big\|_{\ell^q} < \infty \}.
\end{equation*}
We write $M_{p,q}(\R^d) := M^0_{p,q}(\R^d)$ for brevity. We have the following embeddings in the standard Besov scale (cf. \cite[Section~1]{Schippa2022}): By the embedding $\ell^{q_1} \hookrightarrow \ell^{q_2}$ for $q_1 \leq q_2$ and Bernstein's inequality, we have
\begin{align*}
M^s_{p,q_1}(\R^d) \hookrightarrow M^s_{p,q_2}(\R^d) \quad (q_1 \leq q_2), \\
M^s_{p_1,q}(\R^d) \hookrightarrow M^s_{p_2,q}(\R^d) \quad (p_1 \leq p_2).
\end{align*}
By Plancherel's theorem, we have 
\begin{equation}
\label{eq:Plancherel}
M_{2,2}(\R^d) = L^2(\R^d).
\end{equation}
Moreover, we have from kernel estimates with $p=1$ and $p=\infty$ and interpolation with \eqref{eq:Plancherel} the estimates
\begin{align*}
M_{p,p'} \hookrightarrow L^p \hookrightarrow M_{p,p} \quad (2 \leq p \leq \infty), \\
M_{p,p} \hookrightarrow L^p \hookrightarrow M_{p,p'} \quad (1 \leq p \leq 2).
\end{align*}
Lastly, we note that
\begin{equation*}
M^{s_1}_{p,q_1}(\R^d) \hookrightarrow M^{s_2}_{p,q_2}(\R^d)
\end{equation*}
provided that $s_1-s_2 > d \big( \frac{1}{q_2} - \frac{1}{q_1} \big) > 0$ as a consequence of H\"older's inequality.

\subsection{Adapted function spaces}
We use $U^p$-/$V^p$-spaces taking values in modulation spaces as iteration spaces. $U^p$-/$V^p$-spaces based on $L^2$-based Sobolev spaces go back to unpublished notes of Tataru in the context of wave maps (cf. \cite{KochTataru2005}). In case the base space is a general Banach space, we refer to \cite{KochTataruVisan2014} (see also the previous work by Hadac--Herr--Koch \cite{HadacHerrKoch2009,HadacHerrKoch2010}). We shall be brief in the following.

Let $\mathcal{Z}$ be the set of finite partitions $-\infty = t_0 < t_1 < \ldots < t_K= \infty$ and let $\mathcal{Z}_0$ be the set of finite partitions $-\infty < t_0 < t_1 < \ldots < t_K \leq \infty$. We consider $U^p$-/$V^p$-spaces taking values in modulation spaces $M_{p,q}(\R^d)$. Denote the value space in the following by $E$.

\begin{definition}
Let $1 < p <\infty$ and $\{ t_k \}_{k=0}^K \in \mathcal{Z}$ and $\{ \phi \}_{k=0}^{K-1} \subseteq E$ with $\sum_{k=0}^{K-1} \| \phi_k \|_E^p = 1$ and $\phi_0 = 0$. The function $a: \R \to E$ defined by \\ $a = \sum_{k=1}^{K} 1_{[t_{k-1},t_k)} \phi_{k-1}$ is said to be a $U^p$-atom. We define the atomic space
\begin{equation*}
U^p(E) = \{ u = \sum_{j=1}^\infty \lambda_j a_j \, : \, a_j: U^p-\text{atom}, \, (\lambda_j) \in \ell^1 \}
\end{equation*}
with norm
\begin{equation*}
\| u \|_{U^p} = \inf \{ \sum_{j=1}^\infty |\lambda_j| \, : u = \sum_{j=1}^\infty \lambda_j a_j, \, (\lambda_j) \in \ell^1, \, a_j:U^p-\text{atom} \}.
\end{equation*}
\end{definition}
Subspaces are considered as in \cite[Proposition~2.2]{HadacHerrKoch2009}. The spaces of $p$-variation were already considered by Wiener \cite{Wiener1979} (see \cite[Definition~4.8]{KochTataruVisan2014}).
\begin{definition}
Let $1 \leq p < \infty$. $V^p(E)$ is defined as normed space of all functions $v: \R \to E$ such that $\lim_{t \to \pm \infty} v(t)$ exists, $v(\infty):= 0$ (this is purely conventional and does not necessarily coincide with the limit), and $v(-\infty) = \lim_{t \to - \infty} v(t)$. The norm is given by
\begin{equation*}
\| v \|_{V^p} = \sup_{ \{t_k\}_{k=0}^K \in \mathcal{Z}} \big( \sum_{k=1}^K \| v(t_k) - v(t_{k-1}) \|_E^p \big)^{\frac{1}{p}}
\end{equation*}
is finite. Let $V^p_-$ denote the closed subspace of $V^p$ with $\lim_{t \to -\infty} v(t) = 0$.
\end{definition}
We have the embeddings (cf. \cite[Lemma~4.13]{KochTataruVisan2014}):
\begin{equation*}
U^p \hookrightarrow V^p_{rc,-} \hookrightarrow U^q
\end{equation*}
for $1<p<q<\infty$.
Recall the duality $(M_{p,q}(\R^d))' \simeq M_{p',q'}(\R^d)$ for $1 <p,q < \infty$, which is established via the dual pairing
\begin{alignat*}{2}
\langle \cdot, \cdot \rangle: M_{p,q}(\R^d) &\times M_{p',q'}(\R^d) &&\to \C \\
(f&,g) &&\mapsto \int_{\R^d} f \overline{g} dx.
\end{alignat*}
We have the following duality:
\begin{theorem}[{\cite[Theorem~4.14]{KochTataruVisan2014}}]
\label{thm:UpVpDuality}
Let $1<p<\infty$. We have
\begin{equation*}
(U^p(E))^* \simeq V^{p'}(E')
\end{equation*}
in the sense that
\begin{equation*}
T: V^{p'}(E') \to (U^p(E))^*, \quad T(v) = B(\cdot,v)
\end{equation*}
is an isometric isomorphism. 
\end{theorem}
We have an explicit description of $B$ for sufficiently regular functions:
\begin{proposition}
Let $1<p<\infty$, $u \in V^1_-$ be absolutely continuous on compact intervals and $v \in V^{p'}(E')$. Then,
\begin{equation*}
B(u,v) = - \int_{-\infty}^\infty \langle u'(t), v(t) \rangle dt.
\end{equation*}
\end{proposition}
Later we rely on computing the $U^p$-norm with the aid of duality:
\begin{equation}
\label{eq:UpNormDuality}
\| u \|_{U^p(E)} = \sup_{v \in V^{p'}(E'): \| v \|_{V^{p'}(E')} = 1 } |B(u,v)|.
\end{equation}
We remark that the spaces can as well be localized to an interval, in which case we write $U^p(I;E)$, $V^p(I;E)$. We furthermore define the space $DU^p(I;E)$:
\begin{equation*}
DU^p(I;E) = \{ f = u' \, : u \in U^p(I;E) \}
\end{equation*}
with the derivative considered in the distributional sense and
\begin{equation*}
\| f \|_{DU^p(I;E)} = \| u \|_{U^p(I;E)}.
\end{equation*}
By Theorem \ref{thm:UpVpDuality}, we have $(DU^p(I;E))^* \simeq V^{p'}(I;E')$ with respect to a bilinear mapping, which for $f \in L^1(I) \hookrightarrow DU^p(I;E)$ is given by
\begin{equation*}
\tilde{B}(f,v) = \int_a^b \langle f(t), v(t) \rangle dt.
\end{equation*}
We adapt $U^p$-/$V^p$-spaces to the linear Schr\"odinger propagation $e^{it \Delta}$ as usual:
\begin{equation}
\label{eq:AdaptedSpaces}
\| u \|_{X^p_\Delta(I;E)} = \| e^{- it \Delta} u \|_{X^p(I;E)}
\end{equation}
with $X \in \{ U;V; DU \}$.

\section{Bilinear refinements}
\label{section:Bilinear}
By Galilean invariance, we can show bilinear estimates with derivative loss only in the low frequency. In the context of Strichartz estimates on tori, we refer to \cite{KillipVisan2016,Bourgain1993}. Starting point is the following linear Strichartz estimate:
\begin{equation}
\label{eq:L4LinearStrichartzEstimate}
\| e^{it \Delta} u_0 \|_{L^4([0,1] \times \R^d)} \lesssim \| u_0 \|_{M^s_{4,2}(\R^d)}.
\end{equation}

\begin{proposition}
\label{prop:BilinearEstimate}
Let $1 \leq K \ll N$ and suppose that \eqref{eq:L4LinearStrichartzEstimate} holds true. Then, we find the following estimate to hold:
\begin{equation*}
\| P_N e^{it \Delta} u_0 P_K e^{it \Delta} v_0 \|_{L^2([0,1] \times \R^d)} \lesssim K^{2s} \| P_N u_0 \|_{M_{4,2}} \| P_K v_0 \|_{M_{4,2}}.
\end{equation*}
\end{proposition}
\begin{proof}
Let $(Q_{K'})_{K'}$ be a family of frequency projections to balls of size $K$ in $\R^d$ whose supports are covering $B(0,2N) \backslash B(0,N/2)$ finitely overlapping. By almost orthogonality, we find
\begin{equation*}
\begin{split}
&\quad \| P_N e^{it \Delta} u_0 P_K e^{it \Delta} v_0 \|_{L^2([0,1] \times \R^d)}^2 \\
 &\lesssim
\sum_{K'} \| P_N Q_{K'} e^{it \Delta} u_0 P_K e^{it \Delta} v_0 \|^2_{L^2([0,1] \times \R^d)} \\
&= \sum_{K'} \| P_N Q_{K'} e^{it \Delta} u_0 \|^2_{L^4([0,1] \times \R^d)} \| P_K e^{it \Delta} v_0 \|_{L^4([0,1] \times \R^d)}^2. 
\end{split}
\end{equation*} 
We apply \eqref{eq:L4LinearStrichartzEstimate} to the second factor and to the first factor after Galilean transform, which yields
\begin{equation*}
\lesssim K^{4s} \sum_{K'} \| Q_{K'} u_0 \|^2_{M_{4,2}} \| P_K v_0 \|_{M_{4,2}}^2 \lesssim K^{4s} \| P_N u_0 \|^2_{M_{4,2}} \| P_K v_0 \|_{M_{4,2}}^2.
\end{equation*}
The ultimate estimate follows by the finitely overlapping property and the definition of the modulation spaces.
\end{proof}
This yields Proposition \ref{prop:BilinearRefinement} by Theorem \ref{thm:LinearSmoothing}.
In the next step we use the transfer principle to derive an estimate for $V^2_\Delta M_{4,2}$-functions.
\begin{proposition}
Let $K,N \in 2^{\mathbb{N}_0}$ and $1 \leq K \ll N$. Suppose that \eqref{eq:L4LinearStrichartzEstimate} holds. Then, we find the following estimate to hold:
\begin{equation}
\label{eq:V2BilinearEstimate}
\| P_N u P_K v \|_{L^2_{t,x}([0,1] \times \R^d)} \lesssim K^{2s} \| P_N u \|_{V^2_\Delta M_{4,2}} \| P_K v \|_{V^2_\Delta M_{4,2}}.
\end{equation}
\end{proposition}
\begin{proof}
By almost orthogonality, we can write
\begin{equation*}
\| P_N u P_K v \|_{L^2}^2 \lesssim \sum_{K'} \| Q_{K'} P_N u P_K v \|^2_{L^2_{t,x}([0,1] \times \R^d)}
\end{equation*}
with $(Q_{K'})_{K'}$ like above. We apply H\"older's inequality to find
\begin{equation*}
\lesssim \sum_{K'} \| Q_{K'} P_N u \|^2_{L^4_{t,x}([0,1] \times \R^d)} \| P_K v \|^2_{L^4_{t,x}([0,1] \times \R^d)}.
\end{equation*}
We write $P_K v = \sum_m a_m g_m$ with $g_m$ a $U^4_\Delta M_{4,2}$-atom:
\begin{equation*}
g_m = \sum_j 1_{I_j^m} e^{it \Delta} f^m_j, \quad \sum_j \| f_j^m \|^4_{M_{4,2}} = 1.
\end{equation*}
Consequently,
\begin{equation*}
\begin{split}
\| P_K v \|_{L^4_{t,x}([0,1] \times \R^d)} &\leq \sum_m |a_m| \| P_K g_m \|_{L^4_{t,x}([0,1] \times \R^d)} \\
&\leq \sum_m |a_m| \big( \sum_j \| P_K e^{it \Delta} f_j^m \|^4_{L^4_{t,x}(I_j^m \times \R^d)} \big)^{\frac{1}{4}} \\
&\lesssim \sum_m |a_m| \big( \sum_j \| f_j^m \|^4_{M_{4,2}} \big)^{\frac{1}{4}} \\
&\lesssim K^s \sum_m |a_m| \lesssim K^s (1+\varepsilon) \| P_K v \|_{U^4_\Delta M_{4,2}}
\end{split}
\end{equation*}
for any $\varepsilon >0$ by choice of $(a_m) \in \ell^1$. Likewise, by an additional Galilean transform, we find
\begin{equation*}
\| Q_{K'} P_N u \|_{L^4_{t,x}([0,1] \times \R^d)} \lesssim K^s \| Q_{K'} P_N u \|_{U^4_\Delta M_{4,2}}.
\end{equation*}
We use the embedding $V^2_\Delta \hookrightarrow U^4_\Delta$ and carry out the square sum over $K'$ to find
\begin{equation*}
\begin{split}
&\quad \sum_{K'} \| Q_{K'} P_N u \|^2_{L^4_{t,x}([0,1] \times \R^d)} \| P_K v \|^2_{L^4_{t,x}([0,1] \times \R^d)} \\
&\lesssim K^{4s} \sum_{K'} \| Q_{K'} P_N u \|^2_{U^4_\Delta M_{4,2}} \| P_K v \|^2_{U^4_\Delta M_{4,2}} \\
&\lesssim K^{4s} \sum_{K'} \| Q_{K'} P_N u \|^2_{V^2_\Delta M_{4,2}} \| P_K v \|^2_{V^2_\Delta M_{4,2}} \\
&\lesssim K^{4s} \| P_N u \|^2_{V^2_\Delta M_{4,2}} \| P_K v \|^2_{V^2_\Delta M_{4,2}}.
\end{split}
\end{equation*}
The proof is complete.
\end{proof}

\section{Local well-posedness in modulation spaces}
\label{section:ProofEnergyCritical}

This section is devoted to the proof of Theorems \ref{thm:AlmostL2Critical} and \ref{thm:LWPEnergyCriticalNLS}. We begin with the proof of Theorem \ref{thm:AlmostL2Critical}, which is carried out via linear Strichartz estimates (cf. \cite[Theorem~1.2]{KeelTao1998}).

\begin{proof}[{Proof~of~Theorem~\ref{thm:AlmostL2Critical}}]
We give the proof of (1) in detail. The key ingredients are still like in \cite{Schippa2022} smoothing and Strichartz estimates. Let $u_0 = f_1 + f_2$ with $f_1 \in M^s_{6,2}(\R)$ and $f_2 \in L^2(\R)$. Then, Theorem \ref{thm:LinearSmoothing} yields
\begin{equation*}
\| U(t) f_1 \|_{L^6([0,T],L^6(\R))} \lesssim \langle T \rangle^{\frac{1}{6}} \| f_1 \|_{M^s_{6,2}(\R)}
\end{equation*}
and by Strichartz estimates we find
\begin{equation*}
\| U(t) f_2 \|_{L^6([0,T],L^6(\R))} \lesssim \| f_2 \|_{L^2(\R)}.
\end{equation*}
Furthermore, since $U(t) (L^2(\R) + M^s_{6,2}(\R)) = L^2(\R) + M^s_{6,2}(\R)$, we find
\begin{equation*}
\| U(t) u_0 \|_{L^\infty([0,T],L^2(\R) + M^s_{6,2}(\R))} \lesssim \|u_0 \|_{L^2(\R) + M^s_{6,2}(\R)}.
\end{equation*}
The nonlinear estimate is concluded by the inhomogeneous Strichartz estimates
\begin{equation*}
\begin{split}
\big\| \int_0^t e^{i(t-s) \Delta} (|u|^4 u)(s) ds \big\|_{L^6([0,T],L^6(\R))} &\lesssim \| |u|^4 u \|_{L^{6/5}([0,T],L^{6/5}(\R))} \\
 &\lesssim \| u \|^5_{L^6([0,T],L^6(\R))}.
 \end{split}
\end{equation*}
Similarly,
\begin{equation*}
\begin{split}
\quad \big\| \int_0^t e^{i(t-s) \Delta} (|u|^4 u)(s) ds \big\|_{L^\infty([0,T],L^2(\R))} &\lesssim \| |u|^4 u \|_{L^{6/5}([0,T],L^{6/5}(\R))} \\
&\lesssim \| u \|^5_{L^6([0,T],L^6(\R))}.
\end{split}
\end{equation*}
This finishes the proof of (1). The difference with the cubic NLS on $\R$ analyzed in \cite{Schippa2022} is that we cannot afford to apply H\"older's inequality in time. This gives the small data constraint. Regarding the claim (2), we note that in two dimensions, $p=q=4$ are sharp Strichartz indices and by Theorem \ref{thm:LinearSmoothing} we have the smoothing estimate
\begin{equation*}
\| U(t) f \|_{L^4([0,T],L^4(\R^2))} \lesssim \langle T \rangle^{\frac{1}{4}} \| f \|_{M^s_{4,2}(\R^2)}
\end{equation*}
for $s > 0$.
\end{proof}
We turn to the proof of Theorem \ref{thm:LWPEnergyCriticalNLS} in earnest.
As iteration space, we consider $X^s = \ell^2_N U^2_\Delta M^s_{4,2}$ for $s>1$ (cf. Section \ref{section:Notations}). We have for the norm
\begin{equation*}
\| u \|_{X^s} = \big( \sum_N N^{2s} \| P_N u \|^2_{U^2_\Delta M_{4,2}} \big)^{\frac{1}{2}}.
\end{equation*}
We let furthermore
\begin{equation*}
\| v \|_{Y^s} = \big( \sum_N N^{2s} \| P_N u \|^2_{V^2_\Delta M_{4,2}} \big)^{\frac{1}{2}}
\end{equation*}
and have the embedding $X^s \hookrightarrow Y^s$.

 With bilinear estimates in adapted function spaces like in \cite{KillipVisan2016} available, we can apply the arguments from \cite{KillipVisan2016} to prove Theorem \ref{thm:LWPEnergyCriticalNLS}.
We have the following analog of \cite[Proposition~4.1]{KillipVisan2016}:
\begin{proposition}
\label{prop:NonlinearEstimates}
Let $d \in \{3, 4\}$, $s > 1$, and $F(u) = \pm |u|^{\frac{4}{d-2}} u$. Then, for any $0 < T \leq 1$, we find the following estimates to hold:
\begin{equation}
\label{eq:NonlinearEstimateI}
\big\| \int_0^t e^{i(t-s) \Delta} F(u(s)) ds \big\|_{X^s([0,T])} \lesssim \| u \|^{\frac{d+2}{d-2}}_{X^s([0,T])}
\end{equation}
and
\begin{equation}
\label{eq:NonlinearEstimateII}
\begin{split}
&\quad \big\| \int_0^t e^{i(t-s) \Delta} [F(u+w)(s) - F(u(s))] ds \big\|_{X^s([0,T])} \\
&\lesssim \| w \|_{X^s([0,T])} \big( \| u \|_{X^s([0,T])} + \| w \|_{X^s([0,T])} \big)^{\frac{4}{d-2}}.
\end{split}
\end{equation}
The implicit constants do not depend on $T$.
\end{proposition}
\begin{proof}
We only have to prove \eqref{eq:NonlinearEstimateII} because \eqref{eq:NonlinearEstimateI} is a special case. By duality, it is enough to show
\begin{equation*}
\begin{split}
&\quad \big| \int_0^T \int_{\R^d} [F(u+w)(t) - F(u)(t)] v(t,x) dx dt \big| \\
 &\lesssim \| v \|_{Y^{-s}([0,T])} \| u \|_{X^s([0,T])} \big( \| u \|_{X^s([0,T])} + \| w \|_{X^s([0,T])} \big)^{\frac{4}{d-2}}.
\end{split}
\end{equation*}
For the above display, it is enough to show
\begin{equation}
\label{eq:FrequencyLocalizedMultilinearEstimate}
\begin{split}
&\quad \sum_{N_0 \geq 1} \sum_{N_1 \geq \ldots \geq N_{\frac{d+2}{d-2}} \geq 1} \big| \int_0^T \int_{\R^d} v_{N_0}(t,x) \prod_{j=1}^{\frac{d+2}{d-2}} u_{N_j}^{(j)}(t,x) dx dt \big| \\
&\lesssim \| v \|_{Y^{-s}} \prod_{j=1}^{\frac{d+2}{d-2}} \| u^{(j)} \|_{X^s([0,T])}.
\end{split}
\end{equation}
The proof of \eqref{eq:FrequencyLocalizedMultilinearEstimate} follows from linear and bilinear Strichartz estimates combined with Bernstein's inequality. We shall only show the variant of the Killip--Vi\c{s}an argument for $d=3$ to avoid redundancy. In the following let $\varepsilon = s-1 > 0$.\\
\textbf{Case I:} $d=3$. By Littlewood--Paley theory, the two highest frequencies have to be comparable.\\
\textbf{Case I.1:} $N_0 \sim N_1 \geq \ldots \geq N_5$: We apply Proposition \ref{prop:BilinearEstimate} to $v_{N_0} u_{N_2}^{(2)}$ and $u_{N_1}^{(1)} u_{N_3}^{(3)}$ and estimate the remaining factors in $L_{t,x}^\infty$. The estimate in $L^\infty_{t,x}$ is not a local smoothing estimate, but due to Bernstein's and the Cauchy-Schwarz inequality:
\begin{equation*}
\| P_N f \|_{L_x^\infty} \lesssim \| P_N f \|_{M_{\infty,1}} \lesssim \| P_N f \|_{M_{4,1}} \lesssim N^{\frac{3}{2}} \| P_N f \|_{M_{4,2}}.
\end{equation*}

 We write $\mathcal{N}_1 = \{(N_0,N_1,\ldots,N_5) : N_0 \sim N_1 \geq \ldots \geq N_5 \}$ for brevity. The estimates yield
\begin{equation*}
\begin{split}
&\quad \sum_{\mathcal{N}_1} \big| \int_0^T \int_{\R^d} v_{N_0}(t,x) u_{N_1}^{(1)}(t,x) \ldots u_{N_5}^{(5)}(t,x) dx dt \big| \\
&\lesssim \sum_{\mathcal{N}_1} \| v_{N_0} u_{N_2}^{(2)} \|_{L^2_{t,x}} \|u_{N_1}^{(1)} u_{N_3}^{(3)} \|_{L^2_{t,x}} \| u_{N_4}^{(4)} \|_{L^\infty_{t,x}} \| u_{N_5}^{(5)} \|_{L^\infty_{t,x}} \\
&\lesssim \sum_{\mathcal{N}_1} N_2^{\frac{1}{2}+\frac{\varepsilon}{2}} N_3^{\frac{1}{2}+\frac{\varepsilon}{2}} N_4^{\frac{3}{2}} N_5^{\frac{3}{2}} \| v_{N_0} \|_{V^2_\Delta M_{4,2}} \prod_{i=1}^5 \| u_{N_i}^{(i)} \|_{V^2_\Delta M_{4,2}} \\
&\lesssim \| v \|_{Y^{-s}} \prod_{i=1}^5 \| u^{(j)} \|_{Y^s}.
\end{split}
\end{equation*}
By the embedding $X^s \hookrightarrow Y^s$ the proof of Case I.1 is complete.\\
\textbf{Case I.2:} $N_0 \lesssim N_1 \sim N_2 \geq N_3 \geq N_4 \geq N_5$. Denote the summation set with $\mathcal{N}_2$. We apply two bilinear estimates to $v_{N_0} u_{N_1}^{(1)}$ and $u_{N_2}^{(2)} u_{N_3}^{(3)}$ and $L^\infty_{t,x}$-estimates to the other factors to find
\begin{equation*}
\begin{split}
&\quad \sum_{\mathcal{N}_2} \big| \int_0^T \int_{\R^d} v_{N_0}(t,x) u_{N_1}^{(1)}(t,x) \ldots u_{N_5}^{(5)}(t,x) dx dt \big| \\
&\lesssim \sum_{\mathcal{N}_2} \| v_{N_0} u_{N_1}^{(1)} \|_{L^2_{t,x}} \| u_{N_2}^{(2)} u_{N_3}^{(3)} \|_{L^2_{t,x}} \| u_{N_4}^{(4)} \|_{L^\infty_{t,x}} \| u_{N_5}^{(5)} \|_{L^\infty_{t,x}} \\
&\lesssim \sum_{\mathcal{N}_2} N_0^{\frac{1}{2}+\frac{\varepsilon}{2}} N_3^{\frac{1}{2}+\frac{\varepsilon}{2}} N_4^{\frac{3}{2}} N_5^{\frac{3}{2}} \| v_{N_0} \|_{V^2_\Delta M_{4,2}} \prod_{i=1}^5 \| u_{N_i}^{(i)} \|_{V^2_\Delta M_{4,2}} \\
&\lesssim \sum_{\mathcal{N}_2} \frac{N_0^{\frac{3}{2}+\frac{3 \varepsilon}{2}} N_4^{\frac{1}{2}-\varepsilon} N_5^{\frac{1}{2}-\varepsilon}}{N_1^s N_2^s N_3^{\frac{s}{2}}} \| v_{N_0} \|_{Y^{-s}} \prod_{i=1}^5 \| u_{N_i}^{(i)} \|_{Y^s} \\
&\lesssim \| v \|_{Y^{-s}} \prod_{i=1}^5 \| u^{(i)} \|_{Y^s}.
\end{split}
\end{equation*}
This finishes the proof of Case $I$. For the details of the proof of Case $II$ for $d=4$ we refer to \cite{KillipVisan2016}.
\end{proof}
We can complete the proof of Theorem \ref{thm:LWPEnergyCriticalNLS} along the lines of \cite{HerrTataruTzvetkov2011,KillipVisan2016} with Proposition \ref{prop:NonlinearEstimates} at hand.
\begin{proof}[Proof of Theorem \ref{thm:LWPEnergyCriticalNLS}]
For small initial data we can construct a solution on $[0,1]$ by showing that
\begin{equation*}
\Phi(u)(t) := e^{it \Delta} u_0 \mp i \int_0^t e^{i(t-s) \Delta} F(u(s)) ds
\end{equation*}
is a contraction mapping within
\begin{equation*}
B = \{ u \in X^s([0,1]) \cap C_t([0,1],M_{4,2}^s(\R^d)) \, : \, \| u \|_{X^s} \leq 2 \eta \}
\end{equation*}
endowed with $d(u,v):= \| u- v \|_{X^s([0,1])}$. This is a consequence of Proposition \ref{prop:NonlinearEstimates} by observing that $\Phi$ maps $B$ into itself by \eqref{eq:NonlinearEstimateI} and $\Phi$ is indeed contracting by \eqref{eq:NonlinearEstimateII}. This proves Theorem \ref{thm:LWPEnergyCriticalNLS} for small data.

\medskip

For large initial data, we argue with a frequency cutoff. Let $u_0 \in M^s_{4,2}(\R^d)$ with
\begin{equation*}
\| u_0 \|_{M^s_{4,2}(\R^d)} \leq A
\end{equation*}
for some $0<A < \infty$. We consider
\begin{equation*}
B= \{ u \in X^s([0,T]) \cap C_t([0,T],M^s_{4,2}(\R^d)) \, : \, \| u \|_{X^s([0,T])} \leq 2A, \quad \| u_{>N} \|_{X^s([0,T])} \leq 2 \delta \}
\end{equation*}
under the metric $d(u,v) := \| u - v \|_{X^s([0,T])}$.\\
First, we see that $\Phi$ indeed maps $B$ to itself:
\begin{equation*}
\begin{split}
\| \Phi(u) \|_{X^s} &\leq \| e^{it \Delta} u_0 \|_{X^s} + \big\| \int_0^t e^{i(t-s) \Delta} F(u_{\leq N}(s)) ds \big\|_{X^s} \\
&\quad + \big\| \int_0^t e^{i(t-s)\Delta} [F(u)(s) - F(u_{\leq N})(s))] ds \big\|_{X^s} \\
&\leq \| u_0 \|_{M^s_{4,2}} + C \| F(u_{\leq N}) \|_{L_t^1 M^s_{4,2}} + C \| u_{\geq N} \|_{X^s} \| u \|^{\frac{4}{d-2}}_{X^s} \\
&\leq A + CT \| u_{\leq N} \|_{L_t^\infty M^s_{4,2}} \| u_{\leq N} \|^{\frac{4}{d-2}}_{L_t^\infty M^s_{\infty,1}} + C (2 \delta) (2A)^{\frac{4}{d-2}} \\
&\leq A + CT N^{\frac{6}{d-2}} (2A)^{\frac{d+2}{d-2}} + C(2 \delta) (2A)^{\frac{4}{d-2}} \leq 2A
\end{split}
\end{equation*}
provided $\delta$ is chosen small enough depending on $A$, and $T$ is chosen small enough depending on $A$ and $N$.

Next, we decompose $F(u) = F_1(u) + F_2(u)$, where
\begin{equation*}
F_1(u) = O(u_{>N}^2 u^{\frac{6-d}{d-2}}) \text{ and } F_2(u) = O(u^{\frac{4}{d-2}}_{\leq N} u).
\end{equation*}
We estimate with the H\"older-like inequality for modulation spaces (cf. \cite[Theorem~4.3]{Chaichenets2018})
\begin{equation*}
\begin{split}
&\quad \| P_{> N} \Phi(u) \|_{X^s} \\
&\leq \| e^{it \Delta} P_{>N} u_0 \|_{X^s} + \big\| \int_0^t e^{i(t-s) \Delta} F_1(u(s)) ds \big\|_{X^s} \\ 
&\quad + \big\| \int_0^t e^{i(t-s) \Delta} F_2(u(s)) ds \big\|_{X^s} \\
&\leq \|P_{> N} u_0 \|_{M^s_{4,2}(\R^d)} + C \| u_{>N} \|_{X^s}^2 \| u \|^{\frac{6-d}{d-2}}_{X^s} + C \| F_2(u) \|_{L_t^1 M^s_{4,2}} \\
&\leq \delta + C(2 \delta) (2 A)^{\frac{6-d}{d-2}} + CT \| u \|_{L_t^\infty M^s_{4,2}} \| u_{\leq N} \|_{L_t^\infty M^s_{\infty,1}}^{\frac{2d}{d-2}} \\
&\leq \delta + C(2 \delta) (2 A)^{\frac{6-d}{d-2}} + CT N^{\frac{2d}{d-2}} (2A)^{\frac{d+2}{d-2}}.
\end{split}
\end{equation*}
We can bound the above by $2 \delta$ provided that $\delta$ is chosen small enough depending on $A$, and $T$ is chosen small enough depending on $A$, $\delta$, and $N$.\\
Next, we prove that $\Phi$ is a contraction. We decompose like above $F=F_1+F_2$ and observe
\begin{equation*}
F_1(u) - F_1(v) = O((u-v) (u_{>N} - v_{>N}) (u^{\frac{6-d}{d-2}} + v^{\frac{6-d}{d-2}}))
\end{equation*}
and
\begin{equation*}
F_2(u) - F_2(v) = O((u-v)(u_{\leq N} + v_{\leq N})^{\frac{4}{d-2}}) + O((u_{\leq N} - v_{\leq N}) (u+v) (u_{\leq N} + v_{\leq N})^{\frac{6-d}{d-2}} ).
\end{equation*}
By the above arguments for $u,v \in B$:
\begin{equation*}
\begin{split}
&\quad d(\Phi(u),\Phi(v)) \\
&\lesssim \| u-v \|_{X^s} ( \| u_{>N} \|_{X^s} + \| v_{>N} \|_{X^s}) (\| u \|_{X^s} + \| v \|_{X^s})^{\frac{6-d}{d-2}} \\
&\quad + \| F_2(u) - F_2(v) \|_{L_t^1 M^s_{4,2}} \\
&\lesssim (4 \delta) (4A)^{\frac{6-d}{d-2}} d(u,v) + T \| u-v \|_{L_t^\infty M^s_{4,2}} (\| u_{\leq N} \|_{L_t^\infty M^s_{\infty,1}} + \| v_{\leq N} \|_{L_t^\infty M^s_{\infty,1}})^{\frac{4}{d-2}} \\
&\quad + T (\| u \|_{L_t^\infty M^s_{4,2}} + \| v \|_{L_t^\infty M^s_{4,2}}) \| u_{\leq N} - v_{\leq N} \|_{L_t^\infty M^s_{\infty,1}} \\
&\qquad \times \big( \| u_{\leq N} \|_{L_t^\infty M^s_{\infty,1}} + \| v_{\leq N} \|_{L_t^\infty M^s_{\infty,1}} \big)^{\frac{6-d}{d-2}} \\
&\lesssim [(4\delta) (4A)^{\frac{6-d}{d-2}} + TN^{\frac{4d}{d-2}} (4A)^{\frac{4}{d-2}}] d(u,v) \leq \frac{1}{2} d(u,v),
\end{split}
\end{equation*}
provided $\delta$ is chosen small enough depending on $A$, and $T$ is chosen small enough depending on $A$ and $N$. This yields uniqueness and analytic dependence of the data-to-solution mapping within $B$. To infer uniqueness in $X^s([0,T]) \cap C_t([0,T],M^s_{4,2}(\R^d))$, we can compare two solutions for the same initial data through a common frequency cutoff chosen high enough. Then, we find these to be coinciding in a ball and hence in $X^s([0,T]) \cap C_t([0,T],M^s_{4,2}(\R^d))$.
\end{proof}

\section*{Acknowledgements}

I would like to thank Yufeng Lu for pointing out an error in a previous version.
Funded by the Deutsche Forschungsgemeinschaft (DFG, German Research Foundation) -- Project-ID 258734477 -- SFB 1173.

\end{document}